%% file: ABFP_H1.tex
\documentclass[12pt]{amsart}

\usepackage{amsfonts}
\usepackage{amsmath}
\usepackage[applemac]{inputenc}
\usepackage[english]{babel}
\usepackage{amsthm}
\usepackage{graphicx}
\usepackage{wrapfig}
\usepackage{color}
\usepackage{hyperref}
\usepackage{bbm}

\usepackage{mathtools}

\usepackage[all,cmtip]{xy}
\usepackage[normalem]{ulem}
\usepackage{cancel}

\usepackage{amssymb}
\usepackage{amsopn}
\usepackage{mathrsfs}
\usepackage{enumerate}

\usepackage{todonotes}

\newcommand{\N}{\mathbb{N}}
\newcommand{\R}{\mathbb{R}}
\newcommand{\C}{\mathbb{C}}
\renewcommand{\H}{\mathbb{H}}
\newcommand{\bH}{\mathbb{H}}
\newcommand{\bN}{\mathbb{N}}
\newcommand{\bR}{\mathbb{R}}

\newcommand{\be}{\begin{equation}}
\newcommand{\ee}{\end{equation}}

\newcommand{\nH}{\nabla_\H }
\newcommand{\distr}{\mathcal{D}}			
\renewcommand{\be}{\begin{equation}}
\renewcommand{\ee}{\end{equation}}

\newcommand{\cK}{\mathcal{K}}

\newcommand{\kX}{\mathfrak{X}}
\newcommand{\cF}{\mathcal{F}}
\newcommand{\cS}{\mathcal{S}}

\DeclareMathOperator{\dive}{div}
\DeclareMathOperator{\Dom}{Dom}

\newcommand{\red}{\color{red}}

\DeclareMathOperator{\idty}{Id}

\DeclareMathOperator{\dom}{Dom}

\newtheorem{thm}{Theorem}[section]
\newtheorem{prop}[thm]{Proposition}    
      
\newtheorem{lem}[thm]{Lemma}         
\theoremstyle{definition}
\newtheorem{defn}[thm]{Definition}
   
\theoremstyle{remark}        
\newtheorem{rem}[thm]{Remark}

\newcommand{\pointH}{\mathring{H}}

\numberwithin{equation}{section}

\title[Point interactions for 3D sub-Laplacians]{Point interactions for 3D sub-Laplacians}
\date{\today}

\author[R.~Adami]{Riccardo Adami$^1$}
\address{$^1$Politecnico di Torino, Dipartimento di Scienze Matematiche ``G.L. Lagrange'', Corso Duca degli
Abruzzi, 24, 10129, Torino, Italy}
\email{riccardo.adami@polito.it}

\author[U.~Boscain]{Ugo Boscain$^2$}
\address{$^2$CNRS, Sorbonne Universit\'e, Inria, Universit\'e de Paris, Laboratoire Jacques-Louis
Lions, Paris, France}
\email{ugo.boscain@upmc.fr}

\author[V.~Franceschi]{Valentina Franceschi$^3$}
\address{$^3$Laboratoire Jacques-Louis
Lions, Sorbonne Universit\'e, Universit\'e de Paris, Inria, CNRS, Paris, France}
\email{franceschiv@ljll.math.upmc.fr}

\author[D.~Prandi]{Dario Prandi$^4$}
\address{$^4$CNRS, L2S, CentraleSup\'elec, Universit\'e Paris-Saclay, France}
\email{dario.prandi@centralesupelec.fr}

\begin{document}

\begin{abstract}
  In this paper we show that, for a sub-Laplacian $\Delta$ on a $3$-dimensional manifold $M$, no point interaction centered at a 
  point $q_0\in M$ exists. When $M$ is complete w.r.t.\ the associated sub-Riemannian structure, this means that $\Delta$ acting on $C^\infty_0(M\setminus\{q_0\})$ is essentially self-adjoint in $L^2(M)$. A particular example is  the standard sub-Laplacian on the Heisenberg group.
  This is in stark contrast with what happens in a Riemannian manifold $N$, whose associated Laplace-Beltrami operator acting on $C^\infty_0(N\setminus\{q_0\})$ is never essentially self-adjoint in $L^2(N)$, if $\dim N\le 3$.
  We then apply this result to the Schr\"odinger evolution of a thin molecule, i.e., with a vanishing  moment of inertia, rotating around its center of mass.
  
\end{abstract}

\maketitle

\section{Introduction}

Let $(M,g)$ be a Riemannian manifold endowed with a smooth volume $\omega$ (one can think, e.g., of the Riemannian volume). The associated \emph{Laplace operator} is the operator on $L^2(M,\omega)$ acting on $C_0^\infty(M)$ and defined by $\Delta_\omega=\dive_\omega\circ\nabla$. Here, $C_0^\infty(M)$ is the space of compactly supported smooth functions on $M$, and $\dive_\omega$ denotes the divergence w.r.t.\ the measure $\omega$ and $\nabla$ is the Riemannian gradient.
A fundamental issue is the essential self-adjointness of $\Delta_\omega$, i.e., whether it admits a unique self-adjoint extension in $L^2(M, \omega)$. Indeed, the essential self-adjointness of $\Delta_\omega$ implies the well-posedeness in $L^2(M,\omega)$ of the Cauchy problems for the heat and Schr\"odinger equations, that read, respectively, 
\begin{equation}
  \begin{cases}
    \partial_t \phi = \Delta_\omega \phi,\\
    \phi|_{t=0} = \phi_0 \in L^2(M,\omega),
  \end{cases}
  \qquad
  \begin{cases}
    i\partial_t \psi = -\Delta_\omega \psi,\\
    \psi|_{t=0} = \psi_0 \in L^2(M,\omega).
  \end{cases}
\end{equation}
Roughly speaking, when $\Delta_\omega$ is not essentially self-adjoint, the above Cauchy problems are not well-defined without additional requirements, as for instance boundary conditions on $\partial M$.

The self-adjointness of $\Delta_\omega$ is related with geometric properties of $(M,g)$, as is evident from the  following classical result.
\begin{thm}\label{t-gaffney}
Let $(M,g)$ be a Riemannian manifold that is complete as metric space, and let $\omega$ be any smooth volume on $M$. Then, $\Delta_\omega$ is essentially self-adjoint in $L^2(M,\omega)$.
\end{thm}
This result is due to Gaffney \cite{Gaffney1954} when $\omega$ is the Riemannian volume. A simpler argument, which generalizes to arbitrary smooth measures, is given by Strichartz \cite{Strichartz1983}.

A simple way to obtain non-complete Riemannian manifolds from a given complete one $(M,g)$,  is by removing a point $q_0\in M$. Considering $\Delta_\omega$ on $M\setminus\{q_0\}$ yields the \emph{pointed Laplace operator} $\mathring{\Delta}_\omega$. We have the following.
\begin{thm}
\label{t-coltello}
Let $(M,g)$ be a Riemannian manifold that is complete as metric space, and let $\omega$ be any smooth volume on $M$. Let  ${\mathring\Delta_\omega}$ be the pointed Laplace operator at $q_0\in M$. Then
$\mathring\Delta_\omega$ is essentially self-adjoint  in $L^2(M,\omega)$ if and only if $n\geq4$.
\end{thm}

The above result for the Euclidean space endowed with the Lebesgue measure is a consequence of \cite[Ex.~4, p.~160]{RS2}, while the case of Riemannian manifolds where $\omega$ is the Riemannian volume is treated in \cite{ColindeVerdiere1982}. {Similar arguments can be applied when $\omega$ is an arbitrary smooth volume.}

Theorem~\ref{t-coltello} is relevant in physics. Indeed, in non-relativistic quantum mechanics, self-adjoint extensions of the pointed Laplace operator can be used to construct potentials concentrated at a point, the so-called point interactions, as, e.g., 
\begin{equation}
\left\{\begin{array}{l}
 i\partial_t\psi=(-\Delta_\omega+\alpha\delta_{q_0}) \psi,~~~\alpha\in\R,\\
\psi(0,q)=\psi_0(q).
\end{array}\right.
  \end{equation}
 Here, $\delta_{q_0}$ is a Dirac-like  potential  representing a point interaction. Dirac $\delta$ and $\delta'$ are widely used in modelling of quantum systems, since Fermi's paper \cite{Fermi1936} up to contemporary applications \cite{Albeverio2000, Adami2001, Albeverio2005} . In this language, Theorem \ref{t-coltello} can be interpreted as the fact that point interactions do not occur in dimension 4 and higher or, equivalently, that single points are seen by Laplace operators only in dimension less or equal than 3.

In this paper we  study the essential self-adjointness of sub-Laplacians, i.e., the generalization of the Riemannian Laplace operators to sub-Riemannian manifolds. Let us briefly introduce this setting. We refer to \cite{ABB, Montgomery2002} for a more detailed treatement.

\subsection{Sub-Riemannian manifold}
A sub-Riemannian structure on a smooth manifold  $M$ is given by a family of smooth vector fields $\{X_1,\ldots,X_m\}\subset \operatorname{Vec}(M)$ satisfying the H\"ormander condition.  Namely, let $\distr=\operatorname{span}\{X_1,\ldots,X_m\}$, pose $\distr^1= \distr$ and recursively define $\distr^s = \distr^{s-1}+[\distr,\distr^{s-1}]$, $s\in \N$, $s\ge 2$. This defines the flag $\distr^1\subset\ldots\subset\distr^s\subset\ldots\subset \operatorname{Vec}(M) $. 
Letting $\distr^s_q = \{X(q)\mid X\in \distr^s\}$, $s\ge 1$, the H\"ormander condition then amounts to the requirement that for any $q\in M$ there exists $r=r(q)$ such that $\distr^{r}_q =T_qM$.
A sub-Riemannian manifold is then defined as the pair $(M,\{X_1,\ldots,X_m\})$. With abuse of notation, we will sometimes denote it by $M$.

On a sub-Riemannian manifold  the distance between two points $q_1,q_2\in M$ is defined by
\begin{multline}
d(q_1,q_2)=\inf\bigg\{ \int_0^1\sqrt{\sum_{i=1}^m u_i(t)^2}dt \: \bigg| \: \gamma:[0,1]\to M,\quad \dot \gamma(t)= \sum_{i=1}^m u_i(t)X_i(\gamma(t)),\\ \gamma(0)=q_0,~~\gamma(1)=q_1,~~ u_i\in L^1([0,1],\R),~~i=1,\ldots, m\bigg\}.
\end{multline}
Owing to the Rashevskii-Chow theorem \cite{ABB}, $(M,d)$ is a metric space inducing on $M$ its original topology.
The set of vector fields $\{X_1,\ldots,X_m\}$ is called a {\em generating frame} and it is a generalization of Riemannian orthonormal frames.
As for the latter, there are different choices of generating frames giving rise to the same metric space $(M,d)$, which is the true intrinsic object. 
For an equivalent definition of sub-Riemannian manifold that does not employ generating frames, see, e.g., \cite{ABB}. 

The above definition includes several geometric structures \cite{ABB}. Indeed, letting $k(q) = \dim(\distr_q)$, it holds that:
\begin{itemize}
\item If $k(\cdot)\equiv n$, one obtains a Riemannian structure.
\item If $k(\cdot)\equiv k<n$, one obtains a {\em classical sub-Riemannian structure}.  In this case, we will identify $\distr\subset \operatorname{Vec}(M)$ with the vector distribution $\bigsqcup_{q\in M} \distr_q\subset TM$. 
\item if $k(\cdot)$ is not constant, one obtains a so-called {\em rank-varying sub-Riemannian structure}. This includes what are usually called almost-Riemannian structures \cite{Agrachev2006, ABB}.
\end{itemize}
Motivated by the above observations, we say that a sub-Riemannian structure is \emph{genuine} if $k(q)<n$ for all $q\in M$.

\begin{rem}\label{r:orthonormal}
In the first two cases above, if $k(\cdot)\equiv m$ the family of (linearly independent) vector fields $\{X_1,\ldots, X_m\}$ is a \emph{global orthonormal frame} for the (sub-)Riemannian structure. Observe that, due to topological restrictions, such a frame does not always exist.
However, if $k(\cdot)$ is locally constant around $q_0\in M$, there always exists a \emph{local orthonormal frame}\footnote{That is, one can find an open neighborhood $U$ of $q_0$ and a family of linearly independent vector fields $\{Y_1,\ldots,Y_{k(q_0)}\}\subset \operatorname{Vec}(M)$, such that $\distr_q = \operatorname{span}\{Y_1(q),\ldots, Y_{k({q_0})}(q)\}$ for any $q\in U$, and that the sub-Riemannian distances defined by $\{Y_1,\ldots, Y_{k(q_0)}\}$ and $\{X_1,\ldots,X_m\}$ coincide on $U$.} around $q_0$. 
%
\end{rem}

In this paper a particular role is played by $3$-dimensional structures. 
\begin{defn}
  Consider a genuine sub-Riemannian structure on a $3$-dimensional manifold $M$. We say that $q\in M$ is a \emph{contact point} if $\distr^2_q = T_qM$. If every point of $M$ is contact, we say that the structure is a $3$-dimensional contact structure.
\end{defn}
In other words, in the genuine $3$-dimensional case, a contact point is a point in which the full tangent space is generated by the vector fields $X_1,\ldots,X_m$ and their first Lie brackets. Since $M$ is $3$-dimensional, contact points  coincide with what in the literature are called \emph{regular} points.

\subsection{Sub-Laplacians}
Let $\{X_1,\ldots,X_m\}$ be a generating frame for the sub-Rieman\-nian structure on $M$.
Given a smooth volume $\omega$ the associated sub-Laplacian acting on $C^\infty_0(M)$ is defined as $\Delta_\omega=\dive_\omega\circ\nabla$ where $\dive_\omega$ is computed with respect to the volume $\omega$ and $\nabla$ is the sub-Riemannian gradient, whose expression is
\begin{equation}
  \nabla\phi=\sum_{i=1}^mX_i(\phi) X_i,~~~\phi\in C^\infty(M).
\end{equation}
Such an operator is intrinsic in the sense that it does not depend on the particular choice of generating frame. We have then,
\begin{equation}\label{eq:sub-lapl}
\Delta_\omega =\sum_{i=1}^mX_i^2 +(\dive_\omega X_i) X_i.  
\end{equation}
Notice the presence of the ``sum of squares" of the vector fields of the generating frame plus some first order terms guaranteeing the symmetry of $\Delta_\omega$ w.r.t.\ the volume $\omega$. 

As a consequence of H\"ormander condition, $\Delta_\omega$ is hypoelliptic \cite{Hormander1967}, and we have the following generalization of Theorem~\ref{t-gaffney}.
\begin{thm}[Strichartz, \cite{Strichartz1986}]
  Let $M$ be a sub-Riemannian manifold that is complete as a metric space, and let $\omega$ be any smooth volume on $M$. Then, $\Delta_\omega$ is essentially self-adjoint on $L^2(M,\omega)$.
\end{thm}

The main object of interest in this paper is the \emph{pointed sub-Laplacian} $\mathring\Delta_\omega$ at a point $q_0\in M$. Similarly to the Riemannian case, this is defined as the sub-Laplacian $\Delta_\omega$ on $M\setminus\{q_0\}$. 

\subsection{Main results}
One of the main features of sub-Riemannian manifolds, is the existence of several natural notions of dimension. Although for Riemannian manifolds these  are all coinciding, this is not the case in genuine sub-Riemannian manifolds. For instance, in the case of a classical sub-Riemannian manifold, some relevant dimensions are:
\begin{itemize}
\item the dimension of the space of admissible velocities $k$,
\item the topological dimension $n$,
\item the Hausdorff dimension $Q$ of the metric space $(M,d)$,
\end{itemize}
where $k<n<Q$, see \cite{Mitchell1985}\footnote{Notice that $Q = \sup_{q\in M} Q(q)$ where $Q(q)$ is the local Hausdorff dimension, which can be computed via the flag $\distr^1_q\subset\ldots\subset\distr^{k(q)}_q=T_qM$. In particular, $Q$ is possibly infinite.}. It is then a natural question to understand which of these dimensions are relevant for essential self-adjointness of the pointed sub-Laplacian.
In particular, since in a 3D contact sub-Riemannian manifold we have $k=2$, $n=3$, $Q=4$, in view of Theorem~\ref{t-coltello}, we focus on pointed sub-Laplacians at contact points of genuine 3D sub-Riemannian manifolds. For these structures we prove the following.
\begin{thm}
\label{t-principale}
Let $M$ be a genuine $3$-dimensional sub-Riemannian manifold that is complete as  metric space, and let $\omega$ be any smooth volume on $M$. Let $q_0\in M$ be a contact point, and $\mathring\Delta_\omega$ be the pointed sub-Laplacian at $q_0$. Then $\mathring\Delta_\omega$, with domain $C^\infty_0(M\setminus\{q_0\})$, is essentially self-adjoint in $L^2(M,\omega).$
\end{thm}
The above result follows from Theorem~\ref{thm:3d-delta}, and shows that, regarding the essential self-adjointness of pointed sub-Laplacians, 3D sub-Riemannian manifolds behave like Riemannian manifolds of dimension at least $4$. This suggests that the relevant dimension for self-adjointness is not the topological one, and that a more suitable candidate seems to be the Hausdorff dimension.

A crucial step in establishing Theorem \ref{t-principale} is the following corresponding result for the celebrated Heisenberg group $\H^1$.
\begin{thm}\label{thm:Heis}
  The operator $(\partial_x-\frac y2 \partial_z)^2+ (\partial_x+\frac x2 \partial_z)^2$ 
on $C^\infty_0(\R^3\setminus\{(0,0,0)\})$ is essentially self-adjoint in $L^2(\R^3,dx\,dy\,dz)$.
\end{thm}

When $q_0$ is not a contact point, or $M$ is of dimension larger than 3, we conjecture that Theorem~\ref{t-principale} still holds. However, our techniques are not easily extended to higher dimensions.
In dimension $2$, classical sub-Riemannian manifolds  do not exist, while for rank varying structures we have two cases.
Either the point $q_0$ is Riemannian and then we can conclude that the pointed Laplace operator is not essentially self-adjoint; or $q_0$ is not Riemannian and in this case we conjecture that the pointed Laplace operator is not essentially self-adjoint as well. However, the techniques necessary to study this case are very different from those developed in this paper and we do not treat this case here.

\subsection{Rotations of a thin molecule}

\begin{figure}
  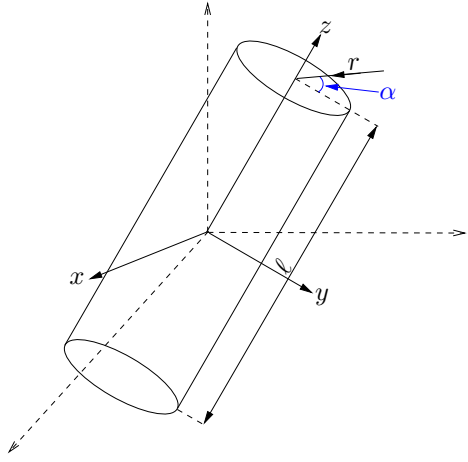
  \caption{The thin molecule is obtained by considering the above rod and letting $r\to 0$. The thin degree of freedom is $\alpha$.}
  \label{fig:rod}
\end{figure}

We now apply Theorem~\ref{t-principale} to the Schr\"odinger evolution on $SO(3)$ of a thin molecule rotating around its center of mass, described as follows. Consider a rod-shaped molecule of mass $m>0$, radius $r>0$, and length $\ell>0$, as in Figure~\ref{fig:rod}. We denote by $z$ the principal axis of the rod, and by $x$ and $y$  two orthogonal ones. Then, the moments of inertia of the molecule are
\begin{equation}
  I_x=I_y=I:=m\frac{3r^2+\ell^2}{12}, \qquad I_z = m\frac{r^2}2.
\end{equation}
Letting $(\omega_x,\omega_y,\omega_z)$ be the angular velocity of the molecule and $L_x=I\omega_x$, $L_y=I\omega_y$, $L_z=I_z\omega_z$ be the corresponding angular momenta, the classical Hamiltonian is 
\begin{equation}
  H=\frac12 \left( I\omega_x^2+I\omega_y^2 +I_z\omega_z^2 \right) = \frac12 \left( \frac{L_x^2}{I}+\frac{L_y^2}{I}+\frac{L_z^2}{I_z} \right).
\end{equation}
Letting $r\to 0$, while keeping  $\ell$ and $m$ constant, we have that $I_z\to 0$, and the classical Hamiltonian reads
\begin{equation}
  H_{\text{thin}}= \frac1{2I} \left( {L_x^2}+{L_y^2} \right).
\end{equation}
The corresponding Schr\"odinger equation is 
\begin{equation}\label{sch-thin}
  i\hbar\frac{d\psi}{dt} = \hat{H}_{\text{thin}}\psi, 
  \qquad 
  \text{where}
  \qquad
  \hat{H}_{\text{thin}} = \frac1{2I} \left( {\hat L_x^2}+{\hat L_y^2} \right).
\end{equation}
Here,  $\hat{L}_x$, $\hat{L}_y$, (and $\hat{L}_z$) are the three angular momentum operators given by (in the following $\alpha,\beta,\gamma$ denote the Euler angles)
\begin{align}
\hat{L}_x=iF_x,~~F_x&=\cos\alpha \cot\beta \frac{\partial}{\partial \alpha}+\sin\alpha\frac{\partial}{\partial \beta}-\frac{\cos\alpha}{\sin\beta}\frac{\partial}{\partial \gamma},\\
\hat{L}_y=iF_y,~~F_y&=\sin\alpha \cot\beta \frac{\partial}{\partial \alpha}-\cos\alpha\frac{\partial}{\partial \beta}-\frac{\sin\alpha}{\sin\beta}\frac{\partial}{\partial \gamma},\\
\hat{L}_z=iF_z,~~F_z&=-\frac{\partial}{\partial \alpha}.
\end{align}
Since $[F_x,F_y]=F_z$ we have that $(SO(3),\{F_x,F_y\})$ is a contact sub-Riemannian manifold. Moreover, being  $SO(3)$ unimodular, we have that  $F_x,F_y,F_z$ are divergence-free with respect to the Haar measure $dh$ (see \cite{Agrachev2009}) and we have that the corresponding sub-Laplacian is 
$$
\Delta_{dh}=F_x^2+F_y^2.
$$
It follows that  $\hat{H}_{\text{thin}}=-\frac{1}{2I}\Delta_{dh}$.

When considering the Schr\"odinger equation \eqref{sch-thin} on functions of $(\alpha,\beta,\gamma)$, we are describing the evolution of a thin molecule in which the thin degree of freedom (i.e., the angle $\alpha$ of the rod w.r.t.\ the $z$ axis) is part of the configuration space.
The essential self-adjointness of the pointed sub-Laplacian $\mathring\Delta_{dh}$ on $SO(3)\setminus\{(\alpha_0,\beta_0,\gamma_0)\}$ given by Theorem~\ref{t-principale} can be interpreted in the following way:
{\em A point interaction centered at $(\alpha_0,\beta_0,\gamma_0)$ does not affect the evolution of a thin molecule}. 

Notice that this would not be the case if the molecule were not thin. Indeed in this case the quantum Hamiltonian would have been  proportional to a left-invariant Riemannian Laplacian on $SO(3)$, and by Theorem~\ref{t-coltello} the elimination of a point from the manifold crashes its essential self-adjointness.

Moreover, if the evolution of the thin molecule is considered on the 2D sphere instead than on $SO(3)$, meaning that we are totally forgetting the thin degree of freedom, then the elimination of a point would break the essential self-adjointness of  
the Laplacian as well.


\subsection{Structure of the paper and strategy of proof}
Sections~\ref{s:prelHeis} and \ref{sec:fa-prelim} are devoted to preliminaries on the Heisenberg group and some of the functional analytic properties of sub-Riemannian manifolds, respectively.
The remaining sections contain the proof of the main result of the paper, Theorem \ref{thm:3d-delta}. This is obtained by first establishing Theorem~\ref{thm:Heis} in Section~\ref{s:Heis}, which is then extended to $3$D genuine sub-Riemannian manifolds in Section~\ref{s:nilpot}.

More precisely, the proof of Theorem~\ref{thm:Heis} consists in first reducing the problem of essential self-adjointness to the absence of $L^2$ solutions of the equation $(\Delta_\omega+i)\theta = \varphi$, where $\varphi$ is a linear combination of derivatives of the Dirac delta mass at $0$, see Lemma~\ref{lem:Pavlov}. This criterion is then verified in Section~\ref{s:nc-fourier} by exploiting the non-commutative Fourier transform associated with the Heisenberg group structure.
Then, in Theorem~\ref{thm:heis-delta-local}, we localize the above result, showing that the self-adjoint extensions of the pointed sub-Laplacian at $0$ defined on a domain $\Omega\subset \H^1$ coincide with those of the (standard) sub-Laplacian on the same domain. 
The latter result is then generalized to any 3D genuine sub-Riemannian manifold via local normal forms, in Section~\ref{s:nilpot}.

Finally, in Appendix~\ref{sec:hardy}, we show how a criterion for essential self-adjointness based on an Hardy inequality with constant strictly bigger than $1$, exploited e.g.\ in  \cite{FPR, Nenciu2008, PRS}, fails for the Heisenberg group. This, in particular, raises a crucial criticism against the results contained in \cite{Yang2013}, and forces us to consider the above strategy of proof for Theorem \ref{thm:3d-delta}.

\section{The Heisenberg group $\H^1$}\label{s:prelHeis}

{\red
}


The Heisenberg group $\mathbb H^1$ is the nilpotent Lie group on $\bR^3$ associated with the non-commutative group law
\begin{equation}
\label{eq:oper}
(x,y,z)*(x',y',z')=\left(x+x',y+y',z+z'+\frac{1}{2}(xy'-x'y)\right).
\end{equation}
The associated Haar measure, i.e., the only (up to multiplicative constant) left-invariant measure on $\H^1$, is the standard Lebesgue measure $\mathcal L^3$ of $\R^3$. One can check that $\H^1$ is unimodular, that is, this measure is also right-invariant. 

A basis for the Lie algebra of left-invariant vector fields is given by
\begin{equation}
\label{eq:vf}
X_\H(x,y,z)=\partial_x-\frac{y}{2}\partial_z,\qquad 
Y_\H(x,y,z)=\partial_y+\frac{x}{2}\partial_z, \qquad 
Z_\H = \partial_z.
\end{equation}
These satisfy the commutator relations $[X_\H,Y_\H]=Z_\H$ and $[X_\H,Z_\H]=[Y_\H,Z_\H]=0$. 
%
The sub-Riemannian structure on $\H^1$ is defined by $\{X_\H,Y_\H\}$. This is a global orthonormal frame, and thanks to the above commutator relations, the sub-Riemannian manifold $(\H^1,\{X_\H,Y_\H\})$ is contact.


We let $\nH$ be the sub-Riemannian gradient of $\H^1$. Then, the  {\em Heisenberg sub-Laplacian} is the associated sub-Laplacian w.r.t.\ the Lebesgue measure, that reads
\begin{equation}
\label{eq:subLapl}
\Delta_\H =X_\H^2+Y_\H^2= \partial_x^2 +\partial_y^2+\frac{x^2+y^2}4\partial_z^2 + (x\partial_y-y\partial_x)\partial_z.  
\end{equation}

\begin{rem}[Fundamental solution]
By \cite[Thm.~2]{Folland1973}, the fundamental solution $\Gamma:\H^1\setminus\{0\}\to \R$  of the operator $-\Delta_\H$ is 
\begin{equation}
\label{eq:Gamma}
\Gamma(p)=\frac{1}{(8\pi) N(p)^2}.
\end{equation}
Here, $N$ is the Koranyi norm (see \cite[Section 2.2.1]{CDPT07}), given by
\begin{equation}
\label{eq:korany}
N(x,y,z)=((x^2+y^2)^2+16 z^2)^{1/4}.
\end{equation}
A simple computation shows that $\Gamma$ is not square-integrable on any compact set containing the origin nor on its complement. This is in contrast with what happens for the fundamental solution of the Euclidean Laplacian on $\R^3$, $\Gamma_{\R^3}(p)= (4\pi |p|)^{-1}$, which is square-integrable near the origin. 
\end{rem}

%



Associated with the group structure of $\H^1$ we have the family of \emph{anisotropic dilations} $\varrho_\lambda:\H^1\to \H^1$, $\lambda>0$, defined by 
\begin{equation}\label{eq:dilation}
  \varrho_\lambda(x,y,z) := (\lambda x,\lambda y, \lambda^2 z).
\end{equation}
One can check that the sub-Riemannian distance from the origin is $1$-homogeneous w.r.t.\ these dilations. Moreover,  we have 
\begin{equation}
  \mathcal{L}^3(\varrho_\lambda(\Omega)) = \lambda^4 \mathcal{L}^3(\Omega), \qquad \forall\lambda>0.
\end{equation}
As a consequence of these facts, the Hausdorff dimension of $\H^1$ is $4$ and $\mathcal H^4=\mathcal L^3$. That is, one more that its topological dimension. 



\section{Sub-Riemannian Sobolev spaces}\label{sec:fa-prelim}

Let $M$ be sub-Riemannian manifold with local generating family $\{X_1,\dots, X_m\}$, endowed with a smooth and positive measure $\omega$. 
We denote by $L^2(M,\omega)$ (or $L^2(M)$) the complex Hilbert space of (equivalence classes of) functions $u:M\to\C$ with scalar product 
\begin{equation}
( u,v)=\int_Mu\,\bar v\, d\omega,\qquad u,v\in L^2(M,\omega),
\end{equation}
where the bar denotes the complex conjugation. The corresponding norm is denoted by $\|u\|^2_{L^2(M)}=( u,u)$. Similarly, $L^2(TM,\omega)$ (or $L^2(TM)$) is the complex Hilbert space of sections of the complexified  tangent bundle $X:M\to T M^\mathbb C$, with scalar product
\begin{equation}
( X,Y)=\int_M g_q(X(q),\overline{Y(q)})\; d\omega(q),\qquad X,Y\in L^2(TM,\omega).
\end{equation}
Here, $g_q$ is the complexification of the scalar product on $\distr_q$ defined by polarization from the norm
\begin{equation}\label{eq:sR-norm}
|\xi|_q^2={\min}\left\{ \sqrt{\sum_{i=1}^m |u_i|^2} \;\bigg|\; \xi = \sum_{i=1}^m u_iX_i(q)\right\}, \qquad \xi\in \distr_q.
\end{equation}
The corresponding norm is $\|X\|^2_{L^2(M)} = (X,X)$. 
Observe that in the above the minimum can be removed if $\{X_1,\ldots,X_m\}$ is an orthonormal frame. In this case, we have
\begin{equation}
  \|X\|^2_{L^2(M)} = \int_M \sum_{i=1}^k |u_i|^2\,d\omega, \quad\text{ where }\quad X = \sum_{i=1}^k u_i X_i.
\end{equation}

Given an open set $\Omega\subset M$, the space $C^\infty_0(\Omega)$ is the space of smooth functions compactly supported in $\Omega$.
For $\Omega\subset M$ we then let $H^2_0(\Omega,\omega)$ (or $H^2_0(\Omega)$) to be the closure of $C^\infty_0(\Omega)$ w.r.t.\ the norm
\begin{equation}\label{eq:Sobolev}
  \|u\|_{H^2(\Omega)}^2=\|u\|_{L^2(\Omega)}^2+\|\Delta_\omega u\|_{L^2(\Omega)}^2,
\end{equation}
where $\Delta_\omega$ is the sub-Laplacian \eqref{eq:sub-lapl}. 
Namely, $H^2_0(\Omega)$ is the domain of the closure of $\Delta_\omega$ as an operator on $L^2(\Omega)$ with domain $C^\infty_0(\Omega)$.
In the following proposition we present a rather classical result ensuring that this space is actually the horizontal second-order Sobolev space.

\begin{prop}\label{prop:sub-ell}
%
  Let $\Omega\subset M$ be open and fix $q\in \Omega$. Then, for any open neighborhood $U\subset \Omega$ of $q$ there exists an open neighborhood $V\subset U$ of $q$ and $C>0$ such that we have, 
  \begin{equation}
    \|\nabla u\|_{L^2(\Omega)}^2+\sum_{i,j=1}^m \|X_iX_j u\|_{L^2(V)}^2\le C \|u\|_{H^2(\Omega)}^2, \qquad u\in C^\infty_0(\Omega).
  \end{equation}
\end{prop}

\begin{proof}
Let $u\in C^\infty_0(\Omega)$.  We start by observing that, since $(\Delta_\omega u,u)=-\|\nabla u\|_{L^2(\Omega)}^2$, by Cauchy-Schwarz and Young's inequality we have
  \begin{equation}\label{eq:est_grad}
    \|\nabla u\|_{L^2(\Omega)}^2\le \|\Delta_\omega u\|_{L^2(\Omega)}\|u\|_{L^2(\Omega)} \le  \frac{1}{2} \|\Delta_\omega u\|_{L^2(\Omega)}^2 + \frac{1}{2}\|u\|_{L^2(\Omega)}^2.
  \end{equation}
The result follows by \cite[Theorem 18.d]{RS76} (see also \cite[Remark 53]{Bramanti14}). In fact, the latter and \eqref{eq:est_grad} yield the existence of an open set $V\subset U$ such that
\begin{equation}
\sum_{i,j=1}^m \|X_iX_j u\|_{L^2(V)}^2 \le C \|u\|_{H^2(U)}^2\leq C \|u\|_{H^2(\Omega)}^2.
\end{equation}
Here, the last inequality follows since $\operatorname{supp}u\subset \Omega$.
\end{proof}

The following result highlights the relevance of $H^2_0(\Omega)$ for our purposes. We recall that a self-adjoint extension of a symmetric operator $A$ on an Hilbert space $\mathcal H$ is a self-adjoint operator $T$ such that $\dom A\subset \dom T$ and $A^*u=Tu$ for any $u\in\dom T$.
\begin{prop}\label{prop:only-closure}
  Let $\Omega\subset M$ be open and fix $p\in \Omega$. 
  Then, if $H^2_0(\Omega)=H^2_0(\Omega\setminus\{p\})$, the self-adjoint extensions of the sub-Laplacian with domain $C^\infty_0(\Omega)$ and $C^\infty_0(\Omega\setminus\{p\})$ coincide.
\end{prop}

\begin{proof}
  By \cite[Theorem~VIII.1]{RS1}, any self-adjoint extension of a given operator is also a self-adjoint extension of its closure, and viceversa. In particular, if two operators have the same closure, then also their self-adjoint extensions coincide. In our case, the assumption $H^2_0(\Omega)=H^2_0(\Omega\setminus\{p\})$ entails exactly that the closure of $\Delta_\omega$ with domains $C^\infty_0(\Omega)$ and $C^\infty_0(\Omega\setminus\{p\})$ are the same.
\end{proof}

In the remainder of the section we derive some essential properties of  $H^2_0(\Omega)$. 

\begin{prop}\label{prop:brezis}
  Let $\Omega\subset \Omega'\subset  M$ be open sets with smooth boundary, and $u\in L^2(\Omega)$. Then, $u\in H^2_0(\Omega)$ if and only if $u_e\in H^2_0(\Omega')$, where
  \begin{equation}
    u_e(p):=
    \begin{cases}
      u(p),&\quad\text{if } p\in \Omega,\\
      0,&\quad\text{otherwise}.
    \end{cases}
  \end{equation}
\end{prop}

\begin{proof}
  The result for the Euclidean Sobolev space of order one is well-known \cite[Proposition~9.18]{Brezis2011}. The same arguments extends in a straightforward way to the case under consideration.
\end{proof}

In view of the above, for any $\Omega\subset\Omega'\subset\bH^1$ we will always identify $H^2_0(\Omega)$ with the set of the corresponding $u_e$ in $H^2_0(\Omega')$. In particular, for any smooth open set $\Omega_1,\Omega_2\subset M$ it holds
\begin{equation}\label{eq:intersection}
    H^2_0(\Omega_1)\cap H^2_0(\Omega_2)=H^2_0(\Omega_1\cap \Omega_2).
\end{equation}

In the sequel we will need the following simple fact, which we will apply with $\Omega_1 = \Omega \setminus \bar B_{\varepsilon/2}(p)$  and $\Omega_2=B_\varepsilon(p)$, where $\Omega$ is a smooth open set, $p\in \Omega$, and $B_r(p)$ stands for the open ball at $p$ of radius $r>0$.

\begin{lem}\label{lem:union-sobolev}
  Let $\Omega_1,\Omega_2 \subset M$ be smooth open sets such that $\partial \Omega_1\cap \partial\Omega_2= \varnothing$ and {$\Omega_1\cap\Omega_2$ is relatively compact in $M$}.
  Then, 
  \begin{equation}
    H^2_0(\Omega_1)+H^2_0(\Omega_2)=H^2_0(\Omega_1\cup\Omega_2).
  \end{equation}
\end{lem}

\begin{proof}
%
  We start by proving the inclusion $H^2_0(\Omega_1)+H^2_0(\Omega_2)\subset H^2_0(\Omega_1\cup\Omega_2)$.
  Let $u_1\in H^2_0(\Omega_1)$ and $u_2\in H^2_0(\Omega_2)$, i.e., there exists $(u^n_i)_n\subset C^\infty_0(\Omega_i)$ such that $u^n_i\rightarrow u_i$ w.r.t.\ the $H^2(\Omega_i)$ norm, $i=1,2$.  Then, the sequence $u_1^n+u_2^n\in C^\infty_0(\Omega_1\cup\Omega_2)$ satisfies
  \begin{equation}
    \|u_1^n+u_2^n-(u_1+u_2)\|_{H^2_0(\Omega_1\cup\Omega_2)} \leq  \|u_1^n-u_1\|_{H^2_0(\Omega_1\cup\Omega_2)}+ \|u_2^n-u_2\|_{H^2_0(\Omega_1\cup\Omega_2)}\stackrel{n\to\infty}{\longrightarrow} 0,
  \end{equation}
which proves the claim.

  We now turn to the other inclusion.
  Let $\chi_1,\chi_2\in C^\infty(\Omega_1\cup\Omega_2)$ be such that $\chi_1+\chi_2=1$, $0\le \chi_i\le 1$ for $i=1,2$, $\chi_1\equiv 1$ on $\Omega_1\setminus\Omega_2$, and $\chi_1\equiv 0$ on $\Omega_2\setminus\Omega_1$. 
  Such smooth functions exist thanks to the fact that $\partial\Omega_1\cap\partial\Omega_2=\varnothing$. 
  Moreover, since $\Omega_1\cap\Omega_2$ is relatively compact, there exists $c>0$ such that $|\nabla\chi_i|\le c$ and $|\Delta_\omega \chi_i|\le c$, $i=1,2$.
  
  Given $u\in H^2_0(\Omega_1\cup\Omega_2)$, let $(u_n)_n\subset C^\infty_0(\Omega_1\cup\Omega_2)$ such that $u_n\rightarrow u$ in $H^2(\Omega_1\cup\Omega_2)$. Then, $\chi_i u_n\rightarrow \chi_i u$ in $L^2(\Omega_i)$, $i=1,2$, and
  \begin{equation}\label{eq:conto}
    \begin{split}
      \|\nabla(\chi_i u_n)-\nabla(\chi_i u)\|_{L^2(\Omega_i)} 
        &\le \|\chi_i(\nabla u_n -\nabla u)\|_{L^2(\Omega_i)} +\||\nabla\chi_i|(u_n-u)\|_{L^2(\Omega_i)}\\
        &\le \|\nabla u_n -\nabla u\|_{L^2(\Omega_1\cup\Omega_2)}+c\|u_n-u\|_{L^2(\Omega_1\cup\Omega_2)}.
    \end{split}
  \end{equation}
On the other hand, by Proposition~\ref{prop:sub-ell} we have $\|\nabla(u_n-u)\|_{L^2(\Omega_1\cup\Omega_2)}\leq C\|u_n-u\|_{H^2_0(\Omega_1\cup\Omega_2)}$, and hence \eqref{eq:conto}  shows that $\nabla(\chi_iu_n)\to \nabla(\chi_iu)$ in $L^2(\Omega_i)$, $i=1,2$.
  Thanks to this fact, a similar computation on $\Delta_\omega(\chi_i u_n)-\Delta_\omega(\chi_i u)$  yields $\chi_i u_n\rightarrow \chi_iu$ in $H^2(\Omega_i)$, $i=1,2$. This proves that $u = \chi_1 u+ \chi_2 u\in H^2(\Omega_1)+H^2(\Omega_2)$.
\end{proof}

\section{Essential self-adjointness of the Heisenberg pointed sub-laplacian}\label{s:Heis}

In this section we focus on the pointed sub-Laplacian in the Heisenberg group. We start by proving Theorem~\ref{thm:Heis} via non-commutative harmonic analysis techniques. 
 We then conclude the section by localizing this result in Theorem~\ref{thm:heis-delta-local}. That is, we show that the self adjoint extensions of the pointed sub-Laplacian on a domain $\Omega\subset \bH^1$ coincide with those of the (standard) sub-Laplacian on the same domain.

\subsection{Deficiency spaces of pointed operators}
In what follows, $\cS(\R^3)$ is the Schwartz space on $\R^3$ and $\cS'(\R^3)$ the space of tempered distributions on $\R^3$. We denote by $\langle T ,u \rangle$ the action of $T\in\mathcal S'(\R^3)$ on $u\in \mathcal S(\R^3)$.
Observe that, given a real symmetric operator $A$ on $L^2(\R^3)$, with $\cS(\R^3)\subset\Dom{A}$ and $A(\cS(\bR^3))\subset \cS(\bR^3)$, its action on $T\in\cS'(\R^3)$ is defined as
\begin{equation}\label{eq:def_sym_distr}
\langle AT,u\rangle=\langle T,Au\rangle\qquad\forall u\in\cS(\R^3).
\end{equation}

The Dirac's delta centered at the origin, is the distribution $\delta_0$ defined as
\begin{equation}
\langle \delta_0, u\rangle=u(0),\qquad u\in  \mathcal S(\R^3).
\end{equation}
For a multi-index $\alpha=(\alpha_1,\alpha_2,\alpha_3)$, we let $D^\alpha = \partial_x^{\alpha_1}\partial_y^{\alpha_2}\partial_z^{\alpha_3}$. Recall that $\langle D^\alpha \delta_0, u \rangle = D^\alpha u(0)$ for $u\in \cS(\R^3)$.

The following lemma is the adaptation to our setting of \cite[Lemma~1]{Pavlov1989}.

\begin{lem}\label{lem:Pavlov}
Let $A$ be a real essentially self-adjoint operator on $L^2(\R^3)$, with domain $\cS(\R^3)$, and $A_0$ be the restriction of $A$ to $C^\infty_0(\R^3\setminus \{0\})$. 
Assume, moreover, that $\operatorname{rng} A\subset \cS(\R^3)$.
%
Then, a function $\theta\in L^2(\R^3)$, $\theta\not\equiv 0$, belongs to the deficiency space $\mathcal K_-(A_0)=\operatorname{Ker}(A_0^*+i)$ of $A_0$ if and only if 
\begin{equation}\label{eq:lem}
(A+i)\theta=\sum_{|\alpha|\le N}c_\alpha D^\alpha\delta_0,
\end{equation}
for some $N\in\N$ and non-identically zero coefficients $(c_\alpha)_{|\alpha|\le N}\subset \mathbb C$.
\end{lem}
\begin{proof}
Let us denote by $(\cdot,\cdot)$ the scalar product on $L^2(\bR^3)$, and recall that $(f, g) = \langle f, \bar g\rangle$ for all $f,g\in L^2(\bR^3)$.
By density of $\Dom{A_0}$, we have that $\theta\in\cK_-({A_0})$ if and only if for every $u\in \Dom {A_0}$ we have
\begin{eqnarray}
0&=\left( ({A_0}^*+i)\theta, u\right)&\quad \text{(density of $\Dom{A_0}$)}\\
&=\left( \theta,({A_0}-i)u\right)&\quad \text{(definition of adjoint)}\\
&=\left( \theta,(A-i)u\right)&\quad \text{(${A_0} u= Au$ if $u\in \Dom{A_0}$)}\\
\label{eq:hypo}&=\left\langle \theta,(A+i)\bar u\right\rangle &\quad \text{($u,Au\in \cS(\R^3)$\text{, and }$A$\text{ is real})}\\
\label{eq:defsp}&=\left\langle (A+i)\theta,\bar u\right\rangle&\quad \text{($A$ is symmetric, see \eqref{eq:def_sym_distr})}.
\end{eqnarray}
Since $\langle D^\alpha \delta_0, v\rangle=0$ for all $v\in \dom A_0\subset \cS(\R^3)$ and $\alpha\in\N^3$, letting $u=\bar v$ in the above we get that $\theta\in \cK_-({A_0})$ if and only if 
\begin{equation}\label{eq:pippo}
\left\langle(A+i)\theta+ D^\alpha\delta_0,v\right\rangle=0,\qquad  {\forall v\in C^\infty_0(\R^3\setminus\{0\})}, \, \forall\alpha\in\N^3.
\end{equation}

By definition of support of a distribution and the density of $C^\infty_0(\R^3\setminus \{0\})$ in the space of $\cS(\R^3)$ functions supported away from $\{0\}$, the latter is equivalent to the fact that the support of the distribution on the left-hand side is contained in $\{0\}$. 
Since the only distributions supported at the origin are finite linear combinations of the Dirac delta and its derivatives, in order to complete the proof it suffices to exclude the case $(A+i)\theta =0$ for $\theta\not\equiv 0$.
This follows since $(A+i)\theta=0$ in the sense of distributions implies that the distribution $A\theta$ belongs to $L^2(\R^3)$. In particular, $\theta\in\dom A^*$ and thus, by essential self-adjointness of $A$, it holds $\theta\in \cK_-(A) = \{0\}$.
\end{proof}

\subsection{Essential self-adjointness of the pointed sub-Laplacian on $\H^1$}\label{s:nc-fourier}

In this section we apply Lemma~\ref{lem:Pavlov} to the sub-Laplacian on $\H^1$, via non-commutative harmonic analysis. This is done in Section~\ref{sss:Heis} and requires some preliminary work that is presented in the next section. 

\subsubsection{Non-commutative harmonic analysis on $\H^1$}
For a classical reference on this topic we refer to \cite{Stein1993}, see also  \cite{Agrachev2009, Fischer2016} for a more recent exposition, and  \cite{Bahouri2016,Bahouri2017} for the definition of the Fourier transform on tempered distributions. See also the pioneering works \cite{Geller1977,Geller1980}.

The Schr\"odinger representations $(\kX^\lambda)_{\lambda\in\bR\setminus\{0\}}$ of $\H^1$ act on $u\in L^2(\bR)$ by
\begin{equation}\label{eq:schrodinger}
  \kX^\lambda_{(x,y,z)} u(\xi):=e^{i\lambda(z-y\xi+\frac{xy}{2})}u(\xi-x),\qquad \xi\in\R,\, (x,y,z)\in \H^1.
\end{equation}
By Stone-Von Neumann Theorem, the dual space $\hat\bH^1$ of $\bH^1$ (i.e., the space of equivalence classes of irreducible representations of $\bH^1$) is the disjoint union of the Schr\"odinger representations and of the Pontryiagin dual of $\bR^2$.
The non-commutative Fourier transform defines an isometry between $L^2(\bH^1)$ and $L^2(\hat \bH^1)$. The latter is the space of operator-valued functions which associate to a representation $\pi\in \hat\bH^1$ (actually, to an equivalence class) an Hilbert-Schmidt operator on its representation space $\mathcal H_\pi$, endowed with the Plancherel measure $d\hat\mu$.

Since it turns out that the measure $d\hat\mu$ is supported only on Schr\"odinger representations acting on $L^2(\R)$,  henceforth, with abuse of notation, we let $\hat\bH^1\simeq \bR\setminus\{0\}$.
Under this identification, the non-commutative Fourier transform of a sufficiently regular function (say, $f\in L^1(\bH^1)\cap L^2(\H^1)$) is the operator-valued map
\begin{equation}
  \hat f^\lambda = \int_{\bH^1} f(p)\,\kX^\lambda_{p^{-1}}\,dp, \qquad\lambda\in\bR\setminus\{0\}.
\end{equation}
In particular, the above defines an Hilbert-Schmidt operator on $L^2(\bR)$ for any $\lambda\neq 0$. Explicitly computing the Plancherel measure $d\hat\mu$ yields
\begin{equation}\label{eq:plancherel}
  \|f\|_{L^2(\bH^1)}^2 = \int_{\bR} \|\hat f^\lambda \|^2_{\operatorname{HS}(L^2(\bR))}\,\frac{|\lambda|}{4\pi}d\lambda, \qquad f\in L^1(\bH^1)\cap L^2(\bH^1).
\end{equation}
It is then standard to extend the above to $f\in L^2(\bH^1)$.

\begin{rem}
  Some differences in the explicit computations of the rest of the section with respect to \cite{Stein1993, Bahouri2016,Bahouri2017} are due to a different choice of group law.
  Indeed, in these papers the following law is considered
\begin{equation}
  (x,y,z)\star (x',y', z') = (x+x',y+y',z+z'-2(xy'-x'y)).
\end{equation}
One can check that $(\bH^1,*)\simeq (\bH^1,\star)$ via the group isomorphism $\Phi(x,y,z)=(\frac x2,-\frac y2,z)$.
\end{rem}

We now let $\cS(\bH^1)=\cS(\R^3)$ be the class of Schwarz functions on $\bH^1$, which are defined as in the Euclidean case.
Direct computations show that, for $f\in \cS(\H^1)$, the action of the Fourier transform of $\Delta_\H f$ on functions  $u\in L^2(\R)$ reduces to
\begin{equation}
  \widehat{(\Delta_\H f)}^\lambda u(\xi) =
 (\partial_\xi^2-\lambda^2 \xi^2)\hat f^\lambda u(\xi), \qquad  \, \xi \in \R.
\end{equation}
For $\lambda\neq0$, an orthonormal basis of $L^2(\bR)$ eigenfunctions for the rescaled harmonic oscillator appearing above is given by the rescaled Hermite functions $(H_{n,\lambda})_{n\in\N}$, which satisfy 
\begin{equation}\label{eq:diagl}
  -\hat\Delta_\H^\lambda H_{n,\lambda} = (2n+1)|\lambda|H_{n,\lambda}, \quad n\in\N. 
\end{equation}
These are defined by $H_{n,\lambda}(\xi) = |\lambda|^{1/4}H_n(|\lambda|^{1/2}\xi)$, where $(H_n)_{n\in\N}$ are the standard Hermite functions.
Motivated by these facts, we define
\begin{equation}\label{eq:tilde-f}
  \tilde f^\lambda (n,m) = (\hat f^\lambda H_{m,\lambda}\mid H_{n,\lambda})_{L^2(\bR)} = \int_{\bH} f(p)\kX_{p^{-1}}(n,m,\lambda)\,dp, \quad n,m\in\bN,
\end{equation}
where we let $p\mapsto \kX_{p}(n,m,\lambda) = (\kX_{p}^\lambda H_{m,\lambda}\mid H_{n,\lambda})_{L^2(\bR)}$ be the coefficient of the representation $\kX^\lambda$ w.r.t.\ $H_{n,\lambda}$ and $H_{m,\lambda}$. 

Let $\tilde \bH^1 = \bN^2\times(\bR\setminus\{0\})$, endowed with the measure $d\tilde w$ defined by
  \begin{equation}
    \int_{\tilde\bH^1} \theta(\tilde w)\,d\tilde w = \sum_{n,m\in\bN} \int_{\bR} \theta(n,m,\lambda)\,\frac{|\lambda|}{4\pi}\,d\lambda,\qquad \tilde w=(n,m,\lambda).
  \end{equation}
  It can be shown, cf.~\cite{Bahouri2016}, that defining $\tilde \cF(f)(n,m,\lambda)=\tilde f^\lambda(n,m)$, yields an isometry $\tilde \cF:L^2(\bH^1)\to L^2(\tilde\bH^1)$.
  In particular,  \eqref{eq:plancherel} is recasted to
\begin{equation}\label{eq:plancherel-tilde}
   \|\tilde{\mathcal{F}}(f)\|_{L^2(\tilde \bH^1)} = \|f\|_{L^2(\bH^1)}, \qquad f\in L^2(\bH^1).
\end{equation}
 In \cite{Bahouri2016,Bahouri2017}, the authours endow $\tilde \H^1$ with a metric, allowing them to define the Schwartz class $\cS(\tilde\H^1)$ and thus the class of tempered distributions $\cS'(\tilde\H^1)$. Similarly to what happens for the standard Fourier transform, it can then be shown that $\tilde\cF$ is a continuous isomorphism between the classes $\cS(\bH^1)$ and $\cS(\tilde\bH^1)$. This allows to extend $\tilde\cF$ to tempered distributions on $\bH^1$, i.e., elements of $\cS'(\bH^1)$, via the following relation
\begin{equation}\label{eq:def-distr}
  \langle \tilde\cF T, \theta\rangle_{\cS'(\tilde\bH^1)} =   \langle T, \tilde\cF^* \theta\rangle_{\cS'(\bH^1)},\qquad \theta\in \cS(\tilde \bH^1),\ T\in \cS'(\bH^1).
\end{equation}
Here $\langle\cdot,\cdot\rangle$ denotes the duality, and $\tilde\cF^*$  is obtained by computing the above for $T\in \cS({\bH}^1)\subset \cS'(\H^1)$. Namely, this yields
\begin{equation}\label{eq:cf-star}
  [\tilde\cF^*\theta](p) = \int_{\tilde \H^1} \theta(\tilde w)\,\kX_{p^{-1}}(\tilde w)\,d\tilde w, \qquad p\in \bH^1.
\end{equation}

We then have the following.
\begin{prop}\label{prop:delta-fourier}
  Let $\delta_0$ be the Dirac distribution on $\bH^1$, centered at the origin. Then, for any multi-index $\alpha\in\bN^3$, and $\lambda\neq 0$ we have 
  	\begin{equation}
  		\tilde \delta_0^\lambda = \idty
    \quad \text{and}\quad
        \widetilde{[D^\alpha \delta_0]}^\lambda = |\lambda|^{{(\alpha_1+\alpha_2)}/{2}} \lambda^{\alpha_3} B_{\alpha},
  	\end{equation}
  	where 
	$B_{\alpha}$ is a non-zero operator on $\ell^2(\N)\times \ell^2(\N)$, such that $B_\alpha(n,m)=0$ if $|n-m|>\alpha_1+\alpha_2$.
\end{prop}

\begin{proof}
  By \eqref{eq:def-distr}, for any $\theta\in \cS(\tilde \bH^1)$, we have $\langle \tilde{\delta}_0,\theta\rangle_{\cS'(\tilde\bH^1)} =\cF^*\theta(0)$. 
  Thus, in order to prove the first part of the statement it suffices to show that
  \begin{equation}
    \tilde\cF^* \theta(0)=\int_{\tilde\bH^1} \theta(n,m,\lambda)\,\delta_{n,m}\,d(n,m,\lambda),
  \end{equation}
  where $\delta_{n,m}$ is the Kroenecker delta. This follows at once from \eqref{eq:cf-star} and the fact that $\kX_0(n,m,\lambda) = (H_{m,\lambda}\mid H_{n,\lambda})= \delta_{n,m}$.

  We now turn to the second part of the statement. Let us remark that,  \begin{equation}\label{eq:P}
    \begin{split}
          \left\langle \widetilde{[D^\alpha \delta_0]}, \theta \right\rangle_{\cS'(\tilde\bH^1)}
    &= (-1)^{|\alpha|}\left\langle  \delta_0, D^\alpha \tilde\cF^*\theta \right\rangle_{\cS'(\bH^1)}\\
    &= (-1)^{|\alpha|}D^\alpha\tilde\cF^*\theta(0) \\
    &= (-1)^{|\alpha|}\int_{\tilde\bH^1} \theta(\tilde w) D^\alpha\left(p\mapsto \kX_{p^-1}(\tilde w)\right)|_{p=0}\,d\tilde w.
    \end{split}
  \end{equation}
To prove the statement, letting $\tilde B_\alpha(n,m,\lambda) = D^\alpha (p\mapsto \kX_{p^{-1}}(n,m,\lambda)|_{p=0}$ we show that $\tilde B_\alpha(n,m,\lambda)= |\lambda|^{{(\alpha_1+\alpha_2)}/{2}} \lambda^{\alpha_3} B_{\alpha}$ for a non-zero operator $B_\alpha$ and that $\tilde B_\alpha(n,m,\lambda)=0$  if $|m-n|>\alpha_1+\alpha_2$.
  By \eqref{eq:schrodinger} and the fact that $\kX_0(n,m,\lambda)=\delta_{nm}$, 
  for any $\alpha_3\in\N$ we have
  \begin{equation}\label{eq:alpha3}
        \tilde B_{(0,0,\alpha_3)}(n,m, \lambda) = (-i\lambda)^{\alpha_3}\kX_{0}(\tilde w) = (-i\lambda)^{\alpha_3} \delta_{n,m}, \quad n,m\in\N.
  \end{equation}   
  This proves the statement for $\alpha=(0,0,\alpha_3)\in \N^3$, for $B_{(0,0,\alpha_3)}=(-i)^{\alpha_3}\delta_{nm}$.
  
  Recall the recurrence relation for Hermite functions
  \begin{equation}
    \lambda\xi H_{m,\lambda}(\xi) = |\lambda|^{1/2}\left(\sqrt{\frac m2} H_{m-1,\lambda}(\xi)  + \sqrt{\frac {m+1}2} H_{m+1,\lambda}(\xi)\right).
  \end{equation}
  This yields that
  \begin{equation}
    \label{eq:YY}
    \tilde B_{(0,1,0)}(n,m, \lambda) = i |\lambda|^{1/2}\left(\sqrt{\frac {n+1}2}\delta_{n+1,m}+\sqrt{\frac{n}2} \delta_{n,m+1}\right), \quad n,m\in\N.
  \end{equation}
  By induction, it is then easy to see that
  \begin{equation}\label{eq:B-comp}
    \tilde B_{(0,\alpha_2,\alpha_3)} = \underbrace{\tilde B_{(0,1,0)}\circ \cdots \circ \tilde B_{(0,1,0)} }_{\alpha_2 \text{ times}}\circ \tilde B_{(0,0,\alpha_3)}.
  \end{equation}
Notice that, by \eqref{eq:YY}, the composition of  $\alpha_2$ copies of $\tilde B_{(0,1,0)}$ is $|\lambda|^{\alpha_2/2}$ times a non-zero operator for any $\alpha_2\in\N$. Moreover, since $\tilde B_{(0,0,\alpha_3)}$ is diagonal, while in $\tilde B_{(0,1,0)}$ the only non-zero elements are the upper and lower diagonal ones, we obtain at once that for any $\alpha_2,\alpha_3\in\mathbb N$ we have
$\tilde B_{(0,\alpha_2,\alpha_3)}(n,m, \lambda)=0$ if $|n-m|>\alpha_2$.
  
  The proof for the case $\alpha=(\alpha_1,\alpha_2,\alpha_3)$ uses similar arguments and the observation that, due to the recurrence relation for Hermite functions
  \begin{equation}
    H'_{m,\lambda}(\xi) =  |\lambda|^{1/2} \left(\sqrt{\frac m2} H_{m-1,\lambda}(\xi) - \sqrt{\frac {m+1}2} H_{m+1,\lambda}(\xi)\right),
  \end{equation}
  we have that
  \begin{equation}
    \label{eq:XX}
    \tilde B_{(1,0,0)}(n,m,\lambda) = |\lambda|^{1/2}\left(\sqrt{\frac {n+1}2}\delta_{n+1,m}-\sqrt{\frac{n}2} \delta_{n,m+1}\right).\qedhere
  \end{equation}
\end{proof}

 \subsubsection{Proof of Theorem \ref{thm:Heis}}\label{sss:Heis}
 Recall that $H=-\Delta_\H$, $\dom{H}=C^\infty_0(\R^3)$ is real and essentially self-adjoint. It is easy to check that $\cS(\R^3)\subset \dom(\bar H)$. We then let $A$ be the real essentially self-adjoint operator obtained by restricting $\bar H$ to $\cS(\R^3)$. 
  Moreover, by smoothness of $X_\H$ and $Y_\H$ it holds $\operatorname{rng} A\subset \cS(\R^3)$. 

 In view of the above, we apply Lemma~\ref{lem:Pavlov} to $A_0=-\Delta_\H$, $\dom(A_0)=C^\infty_0(\R^3\setminus\{0\})$. 
 Assume by contradiction that there exists $\theta\in \cK_-(A_0)$, $\theta\not\equiv 0$. 
Then, there exist non-identically zero coefficients $(c_\alpha)_{|\alpha|\le N}\subset \C$, such that 
\begin{equation}
  (-\Delta_\H+i) \theta = \sum_{|\alpha|\le N}c_\alpha D^\alpha\delta_0.
\end{equation}
We apply $\tilde \cF$ on both sides.
  By \eqref{eq:diagl} we have
  \begin{equation}
    \label{eq:sublap-ft}
    \widetilde{[(-\Delta_\H+i) \theta]}^\lambda(n,m) = \left(|\lambda|(2n+1) +i\right)\tilde \theta^\lambda(n,m).
  \end{equation}
By Proposition~\ref{prop:delta-fourier} this yields
\begin{equation}
  \tilde \theta^\lambda(n,m) = \frac{Q_{n,m}(|\lambda|^{1/2})}{|\lambda|(2n+1)+i} 
  \quad \text{where}\quad
  Q_{n,m}(\lambda) = \sum_{|\alpha|\le N} c_\alpha |\lambda|^{{(\alpha_1+\alpha_2)}/{2}} \lambda^{\alpha_3} B_{\alpha}(n,m).
\end{equation}

We now claim that there exists $n_0,m_0\in \N$, $\Lambda>0$, and $C>0$ such that
\begin{equation}\label{eq:claim1}
  |\tilde\theta^\lambda(n_0,m_0)|^2 \ge \frac{C}{|\lambda|^2(2n+1)^2+1}, \qquad \forall |\lambda|>\Lambda.
\end{equation}
Indeed, it is immediate to guarantee the existence of $n_0,m_0\in\N$ such that $Q_{n_0,m_0}\not\equiv 0$, since otherwise $\tilde \theta^{\lambda}(n,m)=0$ for all $(\lambda,n,m)\in\tilde\H^1$ and hence $\theta\equiv 0$.  
 In this case $Q_{n_0,m_0}$ is a non-zero polynomial in $|\lambda|^{1/2}$ for $\lambda>0$ (resp.\ $\lambda<0$). This yields at once the existence of $C,\Lambda>0$ such that $|Q_{n_0,m_0}(\lambda)|\ge C$ for $\lambda>\Lambda$, and hence the claim.
 
By the previous claim and \eqref{eq:plancherel-tilde} we finally obtain that
\begin{equation}
  \|\theta\|^2_{L^2(\bH^1)} = \|\tilde\theta\|^2_{L^2(\tilde\bH^1)} \ge \int_{|\lambda|>\Lambda} \frac{C}{|\lambda|^2(2n_0+1)^2+1}\,\frac{|\lambda|}{4\pi}d\lambda \gtrsim \int_{1}^{+\infty} \frac{d\lambda}{|\lambda|} = +\infty,
\end{equation}
which is a contradiction.
This implies that $\cK_-(A_0)=\{0\}$ and, since $A_0$ is non-negative, that $A_0$ is essentially self-adjoint. (See, e.g., \cite[Thm.~X.I and Corollary]{RS2})

The statement follows by observing that the essential self-adjointness of $A_0$ implies the essential self-adjointness of $\pointH$, since $\dom{A_0}\subset \dom{\pointH}$.
\qed

\subsection{Heisenberg pointed sub-Laplacian on domains}

In this section, we localize Theorem~\ref{thm:Heis}, by proving the following result.

\begin{thm}\label{thm:heis-delta-local}
  Let $\Omega\subset\H^1$ be an open set with smooth boundary. Then, for any $p\in\Omega$, the set of self-adjoint extensions of $\pointH = -\Delta_\H$ with domain $\dom(\pointH)=C^\infty_0(\Omega\setminus\{p\})$ coincides with the one of $H = -\Delta_\H$ with domain $\dom(H)=C^\infty_0(\Omega)$.
\end{thm}

\begin{proof}
  By Proposition~\ref{prop:only-closure}, it suffices to prove that  $H^2_0(\Omega)=H^2_0(\Omega\setminus\{p\})$.
Thanks to the invariance under left translation of $\Delta_\H$ we can assume $p=0$. Moreover, for any $\varepsilon>0$ we let $U_\varepsilon\subset \R^3$ to be the Euclidean ball of radius $\varepsilon$ centered at the origin.

  By Lemma~\ref{lem:union-sobolev}, for any $\varepsilon>0$ sufficiently small, we have that
  \begin{equation}\label{eq:omega-split}
    H^2_0(\Omega) = H^2_0(U_\varepsilon)+H^2_0(\Omega\setminus U_{\varepsilon/2}) 
    \quad\text{and}\quad
    H^2_0(\Omega\setminus\{0\}) = H^2_0(U_\varepsilon\setminus\{0\})+H^2_0(\Omega\setminus U_{\varepsilon/2}).
  \end{equation}
  Thus, we are reduced to show that $H^2_0(U_\varepsilon)=H^2_0(U_\varepsilon\setminus\{0\})$.

  Since the other inclusion is obvious, let us prove that $H^2_0(U_\varepsilon)\subset H^2_0(U_\varepsilon\setminus\{0\})$. To this aim, we consider $u\in H^2_0(U_\varepsilon)$ and a sequence  $(u_n)_n\subset C^\infty_0(U_\varepsilon)$ such that $u_n\rightarrow u$ in the $H^2$ norm. Observe that, by Theorem~\ref{thm:Heis} and Lemma~\ref{lem:union-sobolev}, we have
  \begin{equation}\label{eq:implies}
    H^2_0(\bH^1) =H^2_0(\bH^1\setminus\{0\})= H^2_0(U_\varepsilon\setminus\{0\})+H^2_0(\bH^1\setminus U_{\varepsilon/2}).
  \end{equation}
Hence, since  $u\in H^2_0(U_\varepsilon)\subset H^2_0(\bH^1)$, \eqref{eq:implies} implies that there exist two sequences $(u^{(1)}_n)_n\subset  C^\infty_0(U_\varepsilon\setminus\{0\})$ and $(u^{(2)}_n)_n\subset  C^\infty_0(\bH^1\setminus U_{\varepsilon/2})$ which converge in the $H^2$ norm respectively to $u^{(1)},u^{(2)}$ where $u^{(1)}+u^{(2)}=u$. Thus, the sequence $u_n-u^{(1)}_n\subset C^\infty_0(U_\varepsilon)$ converges to $u^{(2)}$ in the $H^2$ norm, and hence $u^{(2)}\in H^2_0(U_\varepsilon)\cap H^2_0(\bH^1\setminus U_{\varepsilon/2})=H^2_0(U_\varepsilon\setminus U_{\varepsilon/2})$. Here, the last equality follows by \eqref{eq:intersection}. As a consequence, we can assume $(u^{(2)}_n)_n\subset C^\infty_0(U_\varepsilon\setminus U_{\varepsilon/2})$. Finally, we have shown that the sequence $u^{(1)}_n+u^{(2)}_n$ is contained in $C^\infty_0(U_\varepsilon\setminus\{0\})$,  and satisfies $\lim_n (u^{(1)}_n+u^{(2)}_n)=u$. This completes the proof.
\end{proof}

\section{Essential self-adjointness of 3D pointed sub-Laplacians}\label{s:nilpot}

Let $M$ be a $3$-dimensional genuine sub-Riemannian manifold, endowed with a smooth and positive measure $\omega$.  Let $p\in M$ be a regular point, and $\{X_1,X_2\}$ be a local orthonormal frame for the sub-Riemannian structure in $U\subset M$, $p\in U$. (See Remark~\ref{r:orthonormal}.) By \eqref{eq:sub-lapl}, we have
\begin{equation}\label{eq:delta-omega-3d}
  \Delta_\omega = X_1^2+X_2^2  + X_0, \quad \text{where}\quad  X_0 = \dive_\omega(X_1)X_1+\dive_\omega(X_2)X_2.  
\end{equation}

The purpose of this section is to prove the following. 
\begin{thm}\label{thm:3d-delta}
  The set of self-adjoint extensions of $-\Delta_\omega$ with domain $C^\infty_0(M\setminus\{p\})$ coincides with the one with domain $C^\infty_0(M)$.
\end{thm}
We remark that, since when $M$ is complete the sub-Laplacian is essentially self-adjoint, the above implies Theorem~\ref{t-principale}.


The idea of the proof is to show that, sufficiently near $p$, the Sobolev space $H^2_0$ associated with the sub-Laplacian \eqref{eq:delta-omega-3d} is equivalent to the one associated with the Heisenberg sub-Laplacian. This will then allow to exploit the results obtained in Theorem \ref{thm:heis-delta-local} locally around $p$. Finally, a localization argument completes the proof.

In order to go on with the above plan, we fix the following set of coordinates around $p$, for which we refer to \cite[Sec.~8.2]{Zhitomirskii1995}. 
We stress that the regularity of $p$ is essential for the existence of these coordinates.

\begin{prop}\label{prop:zito}
  There exists a local set of coordinates in a neighborhood $V$ around $p$ such that, denoting by $X_\bH$ and $Y_\bH$ the Heisenberg vector fields, there exists $C=(c_{ij})\in C^\infty(V,\operatorname{GL}_2(\R))$ such that $C^{-1}\in C^\infty(V,\operatorname{GL}_2(\R))$ and
  \begin{equation}
    X_1= c_{11}X_\bH+c_{12}Y_\bH, \qquad X_2 = c_{21} X_\bH + c_{22}Y_\bH.
  \end{equation}
\end{prop}


Since, without loss of generality, we can assume that these coordinates cover the whole $U$, we will henceforth identify points in $U$ with their coordinate representation and $p$ with the origin.
In the following, for $\varepsilon>0$  we let $U_\varepsilon$ be the Euclidean ball centered at $0$ of radius $\varepsilon$, 
and $H^2_0(U_\varepsilon)$ be the Sobolev space \eqref{eq:Sobolev} with respect to the sub-Laplacian $\Delta_\omega$ in $M$.

\begin{prop}\label{prop:local-friedrichs}
  It holds that $H^2_0(U_\varepsilon)=H^2_0(U_\varepsilon\setminus\{p\})$.
\end{prop}

\begin{proof}
  Let us consider the coordinates given by Proposition~\ref{prop:zito}. We denote by $H^2_0(U_\varepsilon,\H^1)$ and $H^2_0(U_\varepsilon\setminus\{p\},\H^1)$ the Sobolev spaces associated with the Heisenberg sub-Laplacian $\Delta_\H=X_\H^2+Y_\H^2$. 
In particular, by Theorem~\ref{thm:heis-delta-local}, $H^2_0(U_\varepsilon,\H^1)=H^2_0(U_\varepsilon\setminus\{p\},\H^1)$. Thus, in order to prove the statement, it suffices to show that the $H^2$ norm w.r.t.\ the sub-Laplacian $\Delta_\omega$ and the measure $\omega$ is equivalent to the one w.r.t. $\Delta_\H$ and the Lebesgue measure.
  
  By smoothness of $\omega$ there exists $\varpi>0$ such that, letting $d\omega = \omega(q)dq$, we have
  \begin{equation}\label{eq:omega-bdd}
    \frac{1}{\varpi} \le \omega(q)\le \varpi, \qquad \forall q\in U_\varepsilon.
  \end{equation}
  In particular, this implies that the $L^2$ norms on $U_\varepsilon$ w.r.t.\ $\omega$ and the Lebesgue measure are equivalent.
  Moreover, by Proposition~\ref{prop:zito}, there exists smooth functions $\alpha_{ij}$ and $\beta_i$, $i,j=1,2$, such that
  \begin{equation}
    \Delta_\H = \sum_{i,j=1}^2 \alpha_{ij}X_iX_j + \sum_{i=1}^2 \beta_i X_i.
  \end{equation}
  Let $c>0$ be such that $|\alpha_{ij}|,|\beta_i|\le c$ on $U_\varepsilon$ for $i,j=1,2$.
  By \eqref{eq:omega-bdd} and Proposition~\ref{prop:sub-ell} with $\Omega=U=U_\varepsilon$, we then have that there exists a constant $C>0$ such that, for any $u\in C^\infty_0(U_\varepsilon)$, it holds
  \begin{equation}
    \begin{split}
    \|\Delta_\H u\|^2_{L^2(U_\varepsilon,\H^1)} 
    &\le \varpi \int_{U_\varepsilon} |\Delta_\H u|^2\,d\omega\\
    &\le c\varpi\left( \sum_{i,j=1}^2\|X_iX_ju\|_{L^2(U_\varepsilon)}^2 +\|\nabla u\|^2_{L^2(U_\varepsilon)} \right)\\
    &\le C \|u\|_{H^2(U_\varepsilon)}.
    \end{split}
  \end{equation}
  The same argument can be used to show that $\|\Delta_\omega u\|_{L^2(U_\varepsilon)}\lesssim \|u\|_{H^2(U_\varepsilon,\H^1)}$, completing the proof of the statement.
\end{proof}

Thanks to the above we are now in a position to complete the proof of the main theorem.

\begin{proof}[Proof of Theorem~\ref{thm:3d-delta}]
  By Proposition~\ref{prop:only-closure}, we need to show that $H^2_0(M)=H^2_0(M\setminus\{p\})$.
  Let $\varepsilon>0$ sufficiently small, so that Proposition~\ref{prop:local-friedrichs} implies that $H^2_0(U_\varepsilon)=H^2_0(U_\varepsilon\setminus\{p\})$. 
  By Lemma~\ref{lem:union-sobolev}, we then have
  \begin{multline*}
      H^2_0(M) 
      = H^2_0(U_\varepsilon)+H^2_0(M\setminus U_{\varepsilon/2})\\
      = H^2_0(U_\varepsilon\setminus\{p\})+H^2_0(M\setminus U_{\varepsilon/2})
      =H^2_0(M\setminus\{p\}).\qedhere
  \end{multline*}
\end{proof}

\appendix
\section{Hardy constant in the Heisenberg group}\label{sec:hardy}

In the Riemannian setting, the essential self-adjointness for $n\ge 4$ follows from the validity of local Hardy-type inequalities, with constant $C_H\ge 1$. Indeed, via normal coordinates, one can show that for every $u\in C_0^\infty(M\setminus\{p\})$ it holds
\begin{equation}
\label{eq:hardy_riem_intro}
\int_M|\nabla_{R}u|\;d\omega\geq
C_H\int_{\{ \delta_{R}<\eta\}}\left(\frac{1}{\delta_{R}^2}-\frac{k}{\delta_{R}}\right)u^2\;d\omega+c\|u\|_{L^2(M)},
\quad
C_H=\left( \frac{n-2}{2}\right)^2,
\end{equation}
for some constants $\eta>0, k\leq 1/\eta, c\in\R$, see \cite{PRS}. Here, $\delta_{R}(q)=d_{R}(q,p)$ is the Riemannian distance from $p\in M$.
By Agmon type estimates, this yields at once the essential self-adjointness for $n\geq 4$ as presented in \cite{Nenciu2008, PRS, FPR}. See \cite[Remark~4.2]{FPR} for a comment on the necessity of the condition $C_H\ge 1$ in order to use this approach. 

In this appendix, we show that Hardy-type inequalities as the above with constant  $C_H\ge 1$ do not hold for $3$-dimensional sub-Riemannian manifolds. In particular,  the Hardy constant in the $3$-dimensional Heisenberg group $\H^1$ is strictly less than $1$, see~\cite{FPHardy} for a discussion on Hardy inequalities in the $(2n+1)$-Heisenberg group. 

In the following, we denote the distance from the origin in $\H^1$ as $\delta(p):=d(p,0)$.
One can check that $\delta$ is smooth on $\R^3\setminus\{x=y=0\}$ and satisfies 
\begin{equation}\label{eq:grad-dist}
|\nH\delta|=1 \quad \text{a.e. on }\H^1.
\end{equation}
Here with abuse of notation $|\cdot|$ denotes the sub-Riemannian norm of horizontal vector fields as defined in \eqref{eq:sR-norm}, where $g=g_\H$ is defined in section \ref{s:prelHeis}.
Observe that, by \cite{Lehrback2017}, there exists $C>0$ such that 
\begin{equation}\label{eq:hardy}
  \int_{\bH^1}|\nabla_\bH u|^2\,dp \ge C \int_{\bH^1}\frac{|u|^2}{\delta^2}\,dp, \qquad \forall u\in C^\infty_0(\bH^1\setminus\{0\}).
\end{equation}
In the sequel we prove the following fact on the sharp constant in the above,  contradicting the result claimed in \cite{Yang2013}.

\begin{thm}\label{thm:hardy}
  We have
  \begin{equation}\label{eq:hardy-const}
    C_H = \inf_{u\in C^\infty_0(\R^3\setminus\{0\})} \frac{\int_{\bH^1}|\nabla_\bH u|^2\,dp}{\int_{\bH^1}\frac{|u|^2}{\delta^2}\,dp}<1.
  \end{equation}
\end{thm}

\begin{rem}
  Numerical computations on the explicit function used in the proof of Theorem~\ref{thm:hardy} yield $C_H\le 0.798$.
\end{rem}

\begin{rem}
  In \cite{Garofalo1990}, an inequality similar to \eqref{eq:hardy} is investigated, where the distance from the origin is replaced by the Koranyi norm $N$, defined in \eqref{eq:korany}. In particular, they obtain
  \begin{equation}
    \int_{\bH^1}|\nabla_\bH u|^2\,dp \ge \int_{\bH^1}|u|^2\frac{|\nabla_\H N|^2}{N^2}\,dp, \qquad \forall u\in C^\infty_0(\bH^1\setminus\{0\}).
  \end{equation}
  Here, the constant is equal to $1$ and is sharp. Unfortunately, since the level sets of $|\nabla_\H N|/N$ are not neighborhoods of the origin, this inequality cannot be paired with the Agmon-type estimates techniques of \cite{FPR,PRS} in order to yield the essential self-adjointness result.
\end{rem}

\subsection{A set of coordinates in $\H^1$}

%
%
%

We define the diffeomorphism $\Phi:\R_+\times \mathbb S^1\times(-2\pi,2\pi)\to \H^1\setminus\{x=y=0\}$ given by
\begin{equation}\label{eq:tilde-phi}
\Phi(t,\theta,r) = 
\left(
\begin{array}{c}
t \frac{\sin(\theta+r)-\sin\theta}r\\
t \frac{-\cos(\theta+r)+\cos\theta}r\\
t^2 \frac{r-\sin r}{2r^2}
\end{array}
\right).
\end{equation}
Observe that $\Phi(t,\theta,r)=\varrho_t\circ\Phi(1,\theta,r)$, where $\varrho_t$ denotes the anisotropic dilation introduced in \eqref{eq:dilation}. 
Moreover, direct computations show that $\Phi^*\mathcal L^3 = t^3\mu(r)dt\,d\theta\,dr$, where
\begin{equation}
\mu(r) = \frac{2-2\cos r-r\sin r}{r^{4}}.
\end{equation}

{Let $\theta_0\in \mathbb S^1$, $h_0\in \R$ and $t\in [0,2\pi/|\lambda_z|]$. Then, the curve $t\mapsto \Phi(t,\theta_0,t h_0)$ is an arc-parametrized length-minimizing curve issuing from $0$.} (See, \cite[Section~4.4.3]{ABB}.) In particular, this implies that $\delta(\Phi(t,\theta,r))=t$, for any $t>0$, $\theta\in\mathbb S^1$ and $r\in[-2\pi,2\pi]$.


One can check that  $\nH \delta = \cos(\theta+r)X_\H+\sin(\theta+r)Y_\H$. Then, we let $(\nH\delta)^\perp = -\sin(\theta+r)X_\H + \cos(\theta+r)Y_\H$ be a choice of horizontal unit vector orthogonal to $\nH\delta$. By \eqref{eq:grad-dist}, $\nabla_\H\delta$ and $(\nabla_\H\delta)^\perp$ form an orthonormal basis of horizontal vector fields.
Straightforward computations then show that, in the coordinates given by $\Phi$, we have the following: 
\begin{gather}
\label{eq:dist}
\nabla_\H\delta = \partial_t + \frac rt\partial_r, \\
(\nabla_\H\delta)^\perp = \frac rt\frac{r-\sin (r)}{r \sin (r)+2 \cos (r)-2} \partial_\theta+ \frac rt w(r) \partial_r.
\end{gather}
Here, we let
\begin{equation}\label{eq:ww}
w(r)=\frac{r}{2-r \cot \left(\frac{r}{2}\right)},\qquad r\in(-2\pi,2\pi).
\end{equation}

\subsection{Preliminary computations on the Koranyi norm}
Recall that the Koranyi norm \eqref{eq:korany} is
\begin{equation}
	N(x,y,z)=({(x^2+y^2)^2+16z^2})^{1/4}.
\end{equation}
In particular, we have
\begin{equation}\label{eq:NPhi}
	N\circ\Phi(t,\theta,r)=\frac{\sqrt{2}t}{r}\sqrt[4]{r^2-2 r \sin (r)-2 \cos
   (r)+2}.
\end{equation}
With a little abuse of notation we still denote by $N$ the Korany norm in the coordinates $\Phi$.
Since $\delta(\Phi(t,\cdot,\cdot))=t$, $t>0$, for any $\alpha\in\bR$ this yields, 
\begin{equation}\label{eq:k-pot}
	\frac{|N^{\alpha/2}|^2}{\delta^2} =
2^{\alpha /2}t^{\alpha-2}\left(\frac{\sqrt[4]{r^2-2 r \sin
   (r)-2 \cos (r)+2}}{r}\right)^{\alpha } = t^{\alpha-2}\gamma_{\alpha}(r).
\end{equation}
Here, $\gamma_{\alpha}$ is defined by the last equality. It is simple to check that $\gamma_\alpha(r)\mu(r)$ is a non-negative and bounded continuous function on $[-2\pi,2\pi]$, whose maximum is $1/12$ at $r=0$ and whose minimum is $0$ at $r=\pm 2\pi$.
In particular, the above is integrable in $t^3\mu(r)\,dt\,dr$ for $t\to+\infty$ only if $\alpha<-2$.

Using the expression of $\nabla_\H\delta$ and $(\nabla_\H\delta)^\perp$, and the fact that they form an orthonormal frame outside $t=0$, we then get
\begin{equation}\label{eq:gradNPhi}
  |\nabla_\H N(t,\theta,r)|^2=\frac{1-\cos (r)}{\sqrt{r^2-2 r \sin (r)-2 \cos (r)+2}}
\end{equation}
In particular, for any $\alpha\in\bR$ we have
\begin{equation}\label{eq:k-grad}
	\begin{split}
	|\nabla_\H(N^{\alpha/2})|^2 
&=\frac{\alpha^2}{4}\frac{r^2 (1-\cos r)}{2 \left(r^2-2 r \sin (r)-2 \cos
   (r)+2\right)} \frac{|N^{\alpha/2}|^2}{\delta^2}\\
&=\frac{\alpha^2}{4}t^{\alpha-2}\gamma_{\alpha}(r)\eta(r).
	\end{split}
\end{equation}
Here, $\eta$ is defined by the last equality, and is independent of $\alpha$.
Observe that, also in this case, the integrability at infinity is true only if $\alpha<-2$.

\subsection{Proof of Theorem~\ref{thm:hardy}}
Let us fix a smooth function $\chi:\mathbb R_+\to [0,1]$ such that $\chi|_{[0,1/2]}\equiv0$ and $\chi|_{[1,+\infty]}\equiv 1$. Then, for $\alpha>-2$, we let
\begin{equation}
  u_\alpha(t,\theta,r) = 
  \begin{cases}
    \chi(t) N^{\alpha/2}(1,\theta,r), &\qquad \text{ if } t\le 1,\\
    N^{\alpha/2}(t,\theta,r), &\qquad \text{ otherwise}.\\
  \end{cases}
\end{equation}
By definition of $\chi$, \eqref{eq:k-pot}, and \eqref{eq:k-grad}, for any $\alpha>-2$ there exists $(v_n)_n\subset C^\infty_0(\H^1)$ such that
\begin{equation}
  \lim_{n\to +\infty}\int_{\H^1} \frac{|v_n|^2}{\delta^2}\,dp = \int_{\H^1} \frac{|u_\alpha|^2}{\delta^2}\,dp ,
  \qquad
  \lim_{n\to +\infty}\int_{\H^1} |\nH v_n|^2\,dp = \int_{\H^1} |\nH u_\alpha|^2\,dp.
\end{equation}
In particular, 
%
\begin{equation}
  C_H \le \inf \left\{ \frac{\int_{\H^1} |\nH u_\alpha|^2\,dp}{\int_{\bH^1} \frac{u_\alpha^2}{\delta^2}} :\: \alpha\in[-3,-2]\right\}.
\end{equation}

Let us estimate the quotient above.
By \eqref{eq:k-pot}, we have
\begin{equation}\label{eq:u-den}
  \int_{\H^1}\frac{|u_\alpha|^2}{\delta^2}\,dp \ge \int_{\delta\ge 1}\frac{|u_\alpha|^2}{\delta^2}\,dp = {2\pi}\int_{1}^{+\infty} t^{\alpha+1}\,dt\int_{-2\pi}^{2\pi}\gamma_\alpha(r)\mu(r)\,dr.
\end{equation}
Observe that the integral in $t$ on the r.h.s.\ goes to $+\infty$ as $\alpha\to -2$.
Moreover, as direct computations show, $N^{\alpha/2}|_{t=1}$ and $\partial_r(N^{\alpha/2})|_{t=1}$ are uniformly bounded from above for $\alpha\in[-2,-3]$. As a consequence, there exists a constant $C>0$ such that $|\nabla_\bH u_\alpha|^2\le C$. In particular, by \eqref{eq:k-grad},
\begin{equation}\label{eq:u-num}
  \int_{\H^1}{|\nH u_\alpha|^2}\,dp \le C \mathcal L^3(\{0\le \delta\le 1\}) + \frac{\alpha^2}4{2\pi} \int_{1}^{+\infty} t^{\alpha+1}\,dt \int_{-2\pi}^{2\pi}\gamma_\alpha(r)\eta(r)\mu(r)\,dr.
\end{equation}
Taking the quotient of \eqref{eq:u-num} and \eqref{eq:u-den}, and passing to the limit as $\alpha\to -2$, we get
\begin{equation}\label{eq:ch}\begin{split}
  C_H 
\le \frac{ \int_{-2\pi}^{2\pi}\gamma_{-2}(r)\eta(r)\mu(r)\,dr}{\int_{-2\pi}^{2\pi}\gamma_{-2}(r)\mu(r)\,dr}.
  \end{split}
\end{equation}
Here, we passed to the limit under the integral signs thanks to monotone convergence.
Simple computations show that $\eta(0)=1$, $\eta(\pm2\pi)=0$, and that $\eta$ is monotone decreasing in $|r|$. Hence, for any $a>0$, it holds
\begin{gather}
\int_{|r|>a}\gamma_{-2}(r)\eta(r)\mu(r)\,dr < \eta(a)\int_{|r|>a}\gamma_{-2}(r)\mu(r)\,dr,\\
\int_{|r|\le a}\gamma_{-2}(r)\eta(r)\mu(r)\,dr \le \int_{|r|\le a}\gamma_{-2}(r)\mu(r)\,dr.
\end{gather}
Summing up, since $\eta(a)<1$ and $\int_{|r|>a}\gamma_{-2}(r)\mu(r)\,dr>0$, by \eqref{eq:ch} we obtain that $C_H<1$.
\qed

\medskip
\noindent \textbf{Acknowledgments.} The authors are grateful to Alessandro Teta for suggesting reference \cite{Pavlov1989}, that led to the  strategy of proof for the essential self-adjointness of the pointed Laplacian on the Heisenberg group. The authors acknowledge that the present research is partially supported by: MIUR Grant Dipartimenti di Eccellenza (2018-2022) E11G18000350001, ANR-15-CE40-0018 project \textit{SRGI - Sub-Riemannian Geometry and Interactions},  ANR-17-CE40-0007 poject \textit{QUACO - Contr\^ole quantique : syst\`emes d'EDPs et applications \`a l'IRM}, G.N.A.M.P.A.\ project \textit{Problemi isoperimetrici in spazi Euclidei e non}. The third author acknowledges the support received from the European Union's Horizon 2020 research and innovation programme under the \emph{Marie Sk\l odowska-Curie grant No 794592}.

\bibliographystyle{abbrv}
\bibliography{biblio}

\end{document}

%% file: moleculaz.pdf_t
\begin{picture}(0,0)%
\includegraphics{moleculaz.pdf}%
\end{picture}%
\setlength{\unitlength}{2368sp}%
\begingroup\makeatletter\ifx\SetFigFont\undefined%
\gdef\SetFigFont#1#2#3#4#5{%
  \reset@font\fontsize{#1}{#2pt}%
  \fontfamily{#3}\fontseries{#4}\fontshape{#5}%
  \selectfont}%
\fi\endgroup%
\begin{picture}(4822,4751)(2699,-6285)
\put(6601,-2566){\makebox(0,0)[lb]{\smash{{\SetFigFont{10}{12.0}{\rmdefault}{\mddefault}{\updefault}{\color[rgb]{0,0,1}$\alpha$}%
}}}}
\put(6271,-2258){\rotatebox{4.0}{\makebox(0,0)[lb]{\smash{{\SetFigFont{10}{12.0}{\rmdefault}{\mddefault}{\updefault}{\color[rgb]{0,0,0}$r$}%
}}}}}
\put(5964,-1869){\makebox(0,0)[lb]{\smash{{\SetFigFont{10}{12.0}{\rmdefault}{\mddefault}{\updefault}{\color[rgb]{0,0,0}$z$}%
}}}}
\put(5926,-4667){\makebox(0,0)[lb]{\smash{{\SetFigFont{10}{12.0}{\rmdefault}{\mddefault}{\updefault}{\color[rgb]{0,0,0}$y$}%
}}}}
\put(3354,-4501){\makebox(0,0)[lb]{\smash{{\SetFigFont{10}{12.0}{\rmdefault}{\mddefault}{\updefault}{\color[rgb]{0,0,0}$x$}%
}}}}
\put(5641,-4414){\rotatebox{60.0}{\makebox(0,0)[lb]{\smash{{\SetFigFont{10}{12.0}{\rmdefault}{\mddefault}{\updefault}{\color[rgb]{0,0,0}$\ell$}%
}}}}}
\end{picture}%